\documentclass[12pt]{article}

\usepackage{amssymb,amsmath,amsthm}
\usepackage{hyperref}

\newcommand{\al}{\alpha}
\newcommand{\Dl}{\Delta}
\newcommand{\ve}{\varepsilon}
\newcommand{\vf}{\varphi}
\newcommand{\sg}{\sigma}

\newcommand{\bC}{\Bbb{C}}
\newcommand{\bR}{\Bbb{R}}
\newcommand{\bN}{\Bbb{N}}

\renewcommand{\r}{\rangle}

\newcommand{\sbs}{\subset}
\newcommand{\os}{\overset}
\newcommand{\us}{\underset}
\newcommand{\wt}{\widetilde}
\newcommand{\RA}{\Rightarrow}
\newcommand{\LRA}{\Leftrightarrow}
\newcommand{\frl}{\forall}

\newcommand{\mes}{\operatorname{mes}}
\newcommand{\dist}{\operatorname{dist}}
\newcommand{\esssup}{\operatornamewithlimits{ess\,sup}}
\newcommand{\sign}{\operatorname{sign}}

\newcommand{\rom}[1]{{\rm #1}}
\newtheorem{theorem}{\ \ \ Theorem}[section]
\newtheorem{lem}{\ \ \ Lemma}[section]

\theoremstyle{remark}
\newtheorem{rem}{\ \ \ Remark}

\numberwithin{equation}{section}

\begin{document}
	
\vspace{1cm} 
{\Large
	\begin{center}
		{\bf Some weighted Hardy-type 
			inequalities and applications\footnote{Published in \href{http://www.rmi.ge/proceedings/volumes/112.htm}{Proc.~A.~Razmadze Math.~Inst.~{\bf112} (1997) 113-132}. This file differs from the published version by minor corrections of English.}}
	\end{center}
}
\vspace{1cm}
\begin{center}
	{\bf Vyacheslav S. Rychkov
	}\\[2cm] 
	{
		Mathematisches Institut, Friedrich--Schiller--Universit\"{a}t Jena,\\
		Ernst--Abbe--Platz 1--4, D--07743 Jena, Germany
	}
	\vspace{1cm}%\today
\end{center}

\vspace{4mm}
\begin{abstract}
	We study the two-weighted estimate
	%	(*)
	\begin{equation}
		\bigg\|\sum_{k=0}^na_k(x)\int_0^xt^kf(t)dt|L_{q,v}(0,\infty)\bigg\|\leq
		c\|f|L_{p,u}(0,\infty)\|,\tag{$*$}
	\end{equation}
	where the functions $a_k(x)$ are not assumed to be positive. It is shown that
	for $1<p\leq q\leq\infty$, provided that the weight $u$ satisfies the
	certain conditions, the estimate
	$(*)$ holds if and only if the estimate
	%	(**)
	\begin{equation}
		\sum_{k=0}^n\bigg\|a_k(x)\int_0^xt^kf(t)dt|L_{q,v}(0,\infty)\bigg\|
		\leq c\|f|L_{p,u}(0,\infty)\|.\tag{$**$}
	\end{equation}
	is fulfilled.
	The necessary and sufficient conditions for $(**)$ to be valid are
	well-known.
	The obtained result can be applied to the estimates of differential
	operators with variable coefficients in some weighted Sobolev spaces.
\end{abstract}

\vspace{.2in}
\vspace{.3in}
\hspace{0.2cm} November 1996

\newpage

\section{Introduction}

        The problem of finding the necessary and sufficient conditions imposed
on the functions $u(x)$, $v(x)$, for which the estimate of the kind
%	(1.1)
\begin{equation}
	\bigg(\int_0^{\infty}\bigg|v(x)\int_0^xA(x,t)f(t)dt\bigg|^qdx\bigg)^
	{1/q}\leq c\bigg(\int_0^{\infty}|u(x)f(x)|^pdx\bigg)^{1/p}
\end{equation}
are valid, attracted great attention for the last decade (see, e.g.,
a wide bibliography in the work of V.~D.~Stepanov \cite{1}). It is well known
that the cases $p>q$ and $p\leq q$ are quite different. Here and
everywhere below we assume that $p\leq q$.

        For $1\leq p\leq q\leq\infty$ and $A(x,t)\equiv1$ the necessary and
sufficient condition for validity of the estimate $(1.1)$ has been obtained by
J.~S.~Bradley \cite{2} and V.~M.~Kokilashvili \cite{3}. The first substantial progress
for $A(x,t)\!\not\equiv\!1$ has been reached in the papers by V.~D.~Stepanov
\cite{4}, \cite{5} and F.~J.~Mar\-tin--Reyes and E.~Sawyer \cite{6} where they investigated the
case $A(x,t)=(x-t)^{\al}$, $\al>0$.

        At present the most general classes of kernels $A(x,t)$ are
apparently considered by R.~Oinarov \cite{7}. The kernels of these classes are
positive and satisfy some additional different types restrictions of which the
most known is ``Oinarov's condition'':
$$
	c_1A(x,t)\leq A(x,y)+A(y,t)\leq c_2A(x,t),\;\;\;t<y<x,\;\;\;c_1,c_2>0.
$$

        We should like to emphasize the following fact. As far we know, until recently no work has been available in which the criteria for the validity of (1.1) with the kernels $A(x,t)$ of alternating
signs would have been considered.

        In the recent work \cite{8} the author has, however, presented some class
of kernels with alternating signs for which we managed to characterize
the admissible weights in (1.1) for $p=q=2$ provided that
the weight $u$ satisfies some extra conditions. In the present work we extend this result to all
$1<p\leq q\leq\infty$ and weaken the conditions on $u$. As an application we
studied in \cite{8} the problem of description of pointwise multipliers in some
weighted Sobolev spaces. As an application here, we solve a more general problem
of finding the criteria of bounded action for differential operators with
variable coefficients from the same weighted Sobolev spaces to the weighted
spaces~$L_q$.

        The paper is organized as follows: the main results are formulated
in \S 2. The proofs are given in \S 3 and \S 4. All the functions are assumed
to be measurable and finite almost everywhere.

%        2.
\section{Results}

        Let $\bR^+=(0,\infty)$ be a half-line, $1\leq p\leq\infty$, and $u$ be
a non-negative function on $\bR$ (weight). Denote by $L_{p,u}$ a weighted space
of functions $f:\bR^+\to\bC$ with a norm
\begin{align*}
	&\|f|L_{p,u}\|=\bigg(\int_0^{\infty}|u(x)f(x)|^pdx\bigg)^{1/p},\;\;\;
	1\leq p<\infty,\\
	&\|f|L_{\infty,u}\|=\esssup_{x>0}|u(x)f(x)|.
\end{align*}
In what follows the latter modification is meant everywhere. We denote
by $p'$ a conjugate exponent: $p'=p/(p-1)$. The main result of the present
paper is the following

%        THEOREM 2.1.
\begin{theorem}
Given $n\in\bN$ and functions $a_k\!:\!\bR^+\!\to\!\bC$,
$k\!=\!0,\dots,n$, let
$1<p\leq q\leq\infty$, $u$, $v$ be two non-negative functions on $\bR^+$.
Assume that there exists a constant $D(u)$ such that
%	(2.1)-(2.2)
\begin{gather}
	\int_0^ru^{-p'}(x)dx<\infty\;\;\;\frl r>0,\\
	\int_0^{2r}x^{(n-1)p'}u^{-p'}(x)dx\leq D(u)\int_0^rx^{(n-1)p'}u^{-p'}
	(x)dx\;\;\;\frl r>0.
\end{gather}
Then the following three assertions are equivalent:

        \rom{(i)} The inequality
        $$
	\bigg\|\sum_{k=0}^na_k(x)\int_0^xt^kf(t)dt|L_{q,v}\bigg\|\leq c
	\|f|L_{p,u}\|
$$
holds;

        \rom{(ii)} the inequalities
$$
	\bigg\|a_k(x)\int_0^xt^kf(t)dt|L_{q,v}\bigg\|\leq c_k\|f|L_{p,u}\|,
	\;\;\;k=0,\dots,n,
$$
hold;

        \rom{(iii)} the values
\begin{gather*}
	S_k=\sup_{r>0}\bigg(\int_r^{\infty}|a_k(x)v(x)|^qdx\bigg)^{1/q}\cdot
	\bigg(\int_0^rx^{-kp'}u^{-p'}(x)dx\bigg)^{1/p'}<\infty,\\
	k=0,\dots,n,
\end{gather*}
are finite.

        If $\wt{c}$ is the best constant in \rom{(i)}, then
$$
	K_1(p,q)\wt{c}\leq\sum_{k=0}^n S_k\leq K_2(D(u),n,p,q)\wt{c}.
$$
\end{theorem}

%        REMARK 1.
\begin{rem}
For $p=q=2$ and under more restrictive condition on $u$ of the
form
$$
	\int_{\Dl}u^{-2}(x)dx\leq D(u)\int_{\frac{1}{2}\Dl}u^{-2}(x)dx,
$$
where $\Dl$ is any interval in $\bR^+$ ($\frac{1}{2}\Dl$ is the twice
smaller interval with the same center), such a theorem has been proved by the
author in~\cite{8}.
\end{rem}

%        REMARK 2.
\begin{rem}
There exist weights $u$ not satisfying $(2.2)$ for which
the assertion of Theorem $2.1$ does not hold. For example, for $u(x)=e^{-x}$
the estimate
$$
	\bigg\|x\int_0^xf(t)dt-\int_0^xtf(t)dt|L_{2,e^{-x}}\bigg\|\leq c\|f|
	L_{2,e^{-x}}\|
$$
is valid, while the corresponding estimates for either summand in
the left-hand side do not hold separately. For details we refer to the author's
work \cite{8}, Prop.~1.2 and also to A.~Kufner \cite{9}, Example 1.
\end{rem}

        As an application of Theorem 2.1 we consider estimates of
differential operators with variable coefficients in weighted Sobolev
spaces. For $l\in\bN$ we distinguish two types of such spaces:
$$
	W_{p,u}^l=\bigg\{f:\bR^+\to\bC\;\big|\;
	\|f|W_{p,u}^l\|=\sum_{k=0}^{l-1}|f^{(k)}(0)|+\|f^{(l)}|L_{p,u}\|
	<\infty\bigg\},
$$
and
\begin{align*}
	\os{\circ}{W}{}_{p,u}^l=&\bigg\{f:\bR^+\to\bC|f^{(k)}(0)=0,\;\;\;
	k=0,\dots,l-1,\|f|\os{\circ}{W}{}_{p,u}^l\|=\|f^{(l)}|L_{p,u}\|<\infty\bigg\}.
\end{align*}
The use of notation $f^{(l)}$ assumes that the function $f$ has an absolutely
continuous derivative $f^{(l-1)}$. Then $f^{(l)}=(f^{(l-1)})'$ exists almost
everywhere. In what follows, the condition (2.1) is assumed to be fulfilled. In
this case $\|f^{(l)}|L_{p,u}\|<\infty$ implies the existence of $f^{(k)}(0)$,
$k=0,\dots,l-1$, which may be understood as the limits  $f^{(k)}(0+0)$, and 
the definitions of the spaces $W_{p,u}^l$ and $\os{\circ}{W}{}_{p,u}^l$
become meaningful. Note that all polynomials of degree $\leq l-1$ belong to
$W_{p,u}^l$. The spaces $W_{p,u}^l$ have been introduced and studied by
L.~D.~Kudryavtsev \cite{10}.

        Let us consider a differential operator of the $l$-th order with
variable coefficients
$$
	P(x,D)=\sum_{m=0}^lb_m(x)D^m,
$$
which acts by the rule
$$
	P(x,D)f(x)=\sum_{m=0}^lb_m(x)f^{(m)}(x).
$$

        Introduce the notation
$$
	d_k(x)=\sum_{m=0}^{l-1-k}b_m(x)\frac{x^{l-1-k-m}}{(l-1-k-m)!},\;\;\;
	k=0,\dots,l-1.
$$

        We then have the following theorem.

%        THEOREM 2.2.
\begin{theorem}
Let $l\in\bN$, $1<p\leq q\leq\infty$, and $u$,$v$ be non-negative on $\bR^+$
functions, and let the condition $(2.1)$ be fulfilled. If $l\geq	2$, then let
the condition $(2.2)$ with $n=l-1$ be also fulfilled.

        \rom{(i)} In order for
$$
	P(x,D):\os{\circ}{W}{}_{p,u}^l\to L_{q,v},
$$
it is necessary and sufficient that the conditions
%	(2.3)
\begin{gather}
	\sup_{r>0}\bigg(\int_r^{\infty}|d_k(x)v(x)|^qdx\bigg)^{1/q}\bigg(
	\int_0^rx^{kp'}u^{-p'}(x)dx\bigg)^{1/p'}<\infty,\\
	k=0,\dots,l-1,\notag
\end{gather}
and
%	(2.4)
\begin{equation}
	b_l(x)\equiv 0\;\;\;\text{for}\;\;p<q,\;\;\|b_lvu^{-1}|L_{\infty}
	(\bR^+)\|<\infty\;\;\text{for}\;\;p=q
\end{equation}
be fulfilled.
\\[-1pt]

        \rom{(ii)} In order that
$$
	P(x,D):W_{p,u}^l\to L_{q,v},
$$
it is necessary and sufficient that the conditions $(2.3)$ and $(2.4)$ be
fulfilled and also that
%	(2.5)
\begin{equation}
	\|P(x,D)x^k|L_{q,v}\|<\infty,\;\;\;k=0,\dots,l-1.
\end{equation}
\end{theorem}

%        REMARK.
\begin{rem}
This theorem generalizes Theorem $2.4$ on pointwise
multipliers in $W_{p,u}^l$ from the author's work $\cite{8}$, which considered the case $p=q=2$, and $P(x,D)f=(\vf f)^{(m)}$, $0\leq m\leq l$.
\end{rem}

        In conclusion let us formulate the following open problem: Find
for $l\geq 2$ the necessary and sufficient conditions on the function
$\vf$ under which
$$
	\|\vf f|W_{2,e^{-x}}^l\|\leq c\|f|W_{2,e^{-x}}^l\|\;\;\;\;\;\frl f\in
	W_{2,e^{-x}}^l.
$$
In other words, it is required to describe pointwise multipliers in the space
$W_{2,e^{-x}}^l$. Theorem 2.4 in \cite{8} does not answer this question. Note that
for $l=1$, $1\leq p\leq\infty$ and for arbitrary weights $u,v$ we can show that
the estimate
$$
	\|\vf f|W_{p,v}^1\|\leq c\|f|W_{p,u}^1\|\;\;\;\;\frl f\in
	W_{p,u}^1
$$
holds if and only if the function $\vf$ satisfies the conditions
\begin{gather*}
	\|\vf vu^{-1}|L_{\infty}(\bR^+)\|<\infty,\\
	\sup_{r>0}\bigg(\int_r^{\infty}|\vf'(x)v(x)|^pdx\bigg)^{1/p}
	\bigg(\int_0^ru^{-p'}(x)dx\bigg)^{1/p'}<\infty.
\end{gather*}

%        3.
\section{Proof of Theorem $2.1$}

%        LEMMA 3.1.
\begin{lem}
Let $X$ be a Banach space, $X^*$ be its conjugate and let $Y\sbs X$
be a closed subspace. Furthermore, let $e\in X$, $e\not\in Y$ be a fixed vector.
Then

        \rom{(i)} $\sup\{\langle y^*,e\r:y^*|_Y=0$,
$\|y^*|X^*\|=1\}=\dist(e,Y)$.

        \rom{(ii)} If there exists a vector $y_0\in Y$ such that
$\dist(e,Y)=\|e-y_0\|$, then there also exists a functional $y_0^*\in X^*$,
$y_0^*|_Y=0$, $\|y_0^*|X^*\|=1$ such that
$$
	\langle y_0^*,e\r=\|e-y_0\|.
$$
\end{lem}

%        REMARK.
\begin{rem}
The symbol $\langle y^*,e\r$ denotes the value of the functional $y^*\!\in\!X^*$ on the
vector $e\in X$. The writing $y^*|_Y=0$ means that $\langle y^*,y\r=0$
$\;\frl y\in Y$.
Item (i) of the above lemma is formulated as Exercise $19$ in Chapter $4$ of
W.~Rudin's book $\cite{11}$.
\end{rem}

%        PROOF.
\begin{proof}
Let $y\in Y$, $y^*|_Y=0$, $\|y^*|X^*\|=1$. We see that
$$
	\langle y^*,e\r=\langle y^*,e-y\r\leq\|e-y\|\leq\dist(e,Y),
$$
whence
%	(3.1)
\begin{equation}
	\sup\{\langle y^*,e\r:y^*|_Y=0,\;\;\;\|y^*\|=1\}\leq\dist(e,Y).
\end{equation}

        Next, since $e\not\in Y$, by the Hahn--Banach theorem (see W.~Rudin
\cite{11}, Th.$3.3$, for every $y\in Y$ there exists $y^*\in X^*$ such that
$\|y^*\|=1$, $y^*|_Y=0$ and the norm of the functional $y^*$ is attained on the
vector $e-y$, i.e.,
$$
	\|e-y\|=\langle y^*,e-y\r=\langle y^*,e\r.
$$
This implies that the inverse inequality in (3.1) is valid and hence item (i)
is fulfilled. Item (ii) is proved simultaneously.
\end{proof}

%        LEMMA 3.2.
\begin{lem}
Let $w$ be  a positive function on $\bR^+$, such that
$$
	\int_0^rw(x)dx<\infty\;\;\;\;\;\frl r>0.
$$
Let $1\leq s<\infty$, $P_{n,r}(x)$ be a polynomial of the $n$-th degree with a
highest degree term $x^n$, satisfying the conditions
$$
	\int_0^r|P_{n,r}(x)|^{s-1}\sign P_{n,r}(x)x^kw(x)dx=0,
	\;\;\;k=1,\dots,n.
$$
Then

        \rom{(i)} all roots of the polynomial $P_{n,r}(x)$ are simple, real and
belong to the interval $(0,r)$.

        \rom{(ii)} Let, in addition, the condition
%	(3.2)
\begin{equation}
	\int_0^{2r}w(x)dx\leq c\int_0^rw(x)dx\;\;\;\;\;\frl r>0
\end{equation}
be fulfilled with  some constant $c=c(w)$. Let $x_1(r)$ be the least
root of the polynomial $P_{n,r}(x)$.
Then
%	(3.3)
\begin{equation}
	\int_0^rw(x)dx\leq K\int_0^{x_1(r)}w(x)dx\;\;\;\;\;\frl r>0,
\end{equation}
for  some constant $K=K(c,n,s)$.
\end{lem}

%        REMARK.
\begin{rem}
The polynomial $P_{n,r}(x)$ does exist for every $n\in\bN$ and $r>0$, and is
the solution of the extremal problem
$$
	\int_0^r|P(x)|^sxw(x)dx\to\us{P}{\min},
$$
where $P(x)$ runs through all polynomials with the highest degree
term $x^n$. See \cite{12}, \S 2.1.1, where also
item (i) is proved (for $w\equiv 1$, but the proof works also in the general
case).
\end{rem}

%        PROOF.
%\begin{proof}
\noindent{\it Proof.}
It remains to prove (ii). Let $x_1,\dots,x_n\in(0,r)$ be the roots of
$P_{n,r}(x)$ enumerated in the increasing order. Suppose also that
$x_{n+1}=r$. We shall prove that
%	(3.4)
\begin{equation}
	\int_0^{x_{m+1}}wdx\leq K_1\int_0^{x_m}wdx\;\;\;\text{for all}\;\;\;
	m=1,\dots,n.
\end{equation}
This will imply the required inequality (3.3) with the constant $K=K_1^n$.

        For $x_{m+1}\leq 4x_m$ the inequality (3.4) follows from (3.2) with
$K_1=c^2$. Let now $x_{m+1}>4x_m$. Taking the polynomial
$$
	R(x)=\prod(x-x_k),
$$
where the product is taken over all $k\in\{1,\dots,n\}\backslash\{m\}$, by the
definition of $P_{n,r}$ we have
$$
	\int_0^r|P_{n,r}|^{s-1}\sign(P_{n,r})Rxwdx=0,
$$
which can be written as
%	(3.5)
\begin{equation}
\ \hskip-1cm	\int_{x_m}^r\!\!\prod_{k=1}^{m-1}|x-x_k|^s|x-x_m|^{s-1}
		\!\prod_{k=m+1}^n\!
	|x-x_k|^sxwdx\!=\!\int_0^{x_m}\!\text{(same integrand)}
\end{equation}
(with obvious modifications for $m=1$ and $m=n$).

        By virtue of simple estimates
\begin{align*}
&\left.\begin{array}{l}
	|x-x_k|\geq x_m,\;\;\;k=1,\dots,m\\
	|x-x_k|\geq x_k/2,\;\;\;k=m+1,\dots,n\\	
	x\geq x_m
      \end{array}\right\}\;\;\;x\in[2x_m,x_{m+1/2}]\\
&\left.\begin{array}{l}
	|x-x_k|\leq x_m,\;\;\;k=1,\dots,m\\
	|x-x_k|\leq x_k,\;\;\;k=m+1,\dots,n\\	
	x\leq x_m
        \end{array}\right\}\;\;\;x\in[0,x_m],
\end{align*}
from (3.5) we have
$$	\int_{2x_m}^{x_{m+1}/2} w\,dx\leq 2^{s(n-m-1)}
		\int_0^{x_m} w\,dx.		$$
Therefore
$$	\int_0^{x_{m+1}}w\leq c\int_0^{x_{m+1}/2}w=c\int_0^{2x_m}w+c
		\int_{2x_m}^{x_{m+1}/2}w\leq(c^2+c 2^{s(n-m-1)})
			\int_0^{x_m}w.  \;\;\qed     $$
%\end{proof}

%        LEMMA 3.3.
\begin{lem}
Let $n$, $a_k(x)$, $k=0,\dots,n$, $p$, $q$, $v(x)$ be as in Theorem
$2.1$. Let a positive on $\bR^+$ function $u(x)$ satisfy the condition
$(2.1)$ and
$$	\int_0^{2r}u^{-p'}(x)dx\leq D\int_0^r u^{-p'}(x)dx\;\;\;\frl r>0  $$
{\rm(}which is weaker than $(2.2))$. Assume the estimate \rom{(i)} with the constant
$c<\infty$ from Theorem $2.1$ holds. Then
%	(3.6)
\begin{gather}
\ \hskip-1cm	S_0\!=\!\sup_{r>0}\bigg(\int_r^\infty\!|a_0(x) v(x)|^q dx
			\bigg)^{1/q}\bigg(\int_0^r\! u^{-p'}(x)dx
			\bigg)^{1/p'}\leq  \notag  \\
	\leq  C(D,n,p)c\!<\!\infty.
\end{gather}
\end{lem}

%        PROOF.
\begin{proof}
Consider the extremal problem
%	(3.7)
\begin{equation}
	\begin{cases}
		\dfrac{\int\limits_0^r f(t)dt}{\big(\int\limits_0^r|u(t)f(t)|^p dt
			\big)^{1/p}}\to \max\limits_f, \\[2mm]
		\displaystyle \int_0^r t^k f(t)dt=0,\;\;\;k=1,\dots,n.
	\end{cases}
\end{equation}
By substitution $f(t)=u^{-p'}(t)g(t)$, this problem reduces to a more convenient
one:
%	(3.8)
\begin{equation}
	\begin{cases}
		\dfrac{\int\limits_0^r g(t)u^{-p'}(t)dt}{\big(\int\limits_0^r|g(t)|^p
			u^{-p'}(t)dt\big)^{1/p}}\to\max\limits_g, \\
		\displaystyle \int_0^r g(t)t^k u^{-p'}(t)dt=0,\;\;\;k=1\dots,n.
	\end{cases}	
\end{equation}

        Let $L_p(A,d\mu)$ denote the space of functions $f:A\to\bC$ with the
norm $(\int_A|f|^p d\mu)^{1/p}$. Let $X=L_{p'}([0,r],u^{-p'}(t)dt)$, $Y$ be
a finite-dimensional subspace $X$ with the basis formed by the functions
$t,t^2,\dots,t^n$. Note
that $X^*\!=\!L_p([0,r], u^{-p'}(t)dt)$, the pairing between
$f\in X$ and $g\in X^*$ being defined by the formula
$$	\langle g,f\rangle=\int_0^r g(t)f(t)u^{-p'}(t)dt.		$$
If we denote by $e$ the vector of the space $X$ representing the function
$e(x)\equiv 1$,
then the extremal problem (3.8) can be written in an abstract manner:
$$	\begin{cases}
		\dfrac{\langle g,e\rangle}{\|g|X^*\|}\to
			\max\limits_{g\in X^*},  \\
		g|_{{}_Y}=0.
	\end{cases}			$$
By Lemma 3.1 (i) we have
%	(3.9)
\begin{equation}
	\sup\bigg\{\frac{\langle g,e\rangle}{\|g|X^*\|}:g|_{{}_Y}=0\bigg\}=
			\dist(e,Y).		
\end{equation}
Moreover, since the subspace $Y$ is finite-dimensional,
$\dist(e,Y)$ is certainly
attained at some vector $y_0\in Y$. Hence by Lemma 3.1 (ii),
$\sup$ is attained in the
right-hand side of (3.9), that is there exists $g_0\in X^*$,
$g_0|_{{}_Y}=0$ such that
%	(3.10)
\begin{equation}
	\frac{\langle g_0,e\rangle}{\|g_0|X^*\|}=
		\min_{c_1,\dots,c_n}\bigg(\int_0^r
		|1+c_1t+\cdots c_nt^n|^{p'} u^{-p'}dt\bigg)^{1/p'}.
\end{equation}
	
        Assume $P(t)=1+c_1t+\cdots+c_nt^n$, $w(t)=u^{-p'}(t)$.
The extremum conditions in the right-hand side of (3.10)
have the form
$$	\int_0^r|P(t)|^{p'-1}\sign P(t)t^k w(t)dt=0, \quad k=1,\dots,n.	$$
Therefore it is clear that the extremal polynomial $\wt{P}(t)$ is a constant multiple of the
polynomial $P_{n,r}(t)$ from Lemma 3.2 (with $s=p'$).
According to that lemma (note that the condition (3.2)
is fulfilled), all the roots of $\wt{P}(t)$
are located on $(0,r)$ and the following inequality holds:
$$	\int_0^r u^{-p'}(t)dt\leq K\int_0^{x_1} u^{-p'}(t)dt	\quad
		\frl r>0,		$$
where $x_1$ is the smallest root and the constant $K$ does not depend on
$r$.
Using this fact, we derive from (3.10) the following estimate:
\begin{gather*}
	\frac{\langle g_0,e\rangle}{\|g_0|X^*\|}=
		\bigg(\int_0^r |\wt{P}(t)|^{p'}u^{-p'}(t)dt\bigg)^{1/p'}= \\
	=\bigg(\int_0^r \prod_{k=1}^n \Big|1-\frac{x}{x_k}\Big|^{p'}
		u^{-p'}(t)dt\bigg)^{1/p'}\geq  \\
	\geq \bigg(\int_0^{x_1} \Big(1-\frac{x}{x_k}\Big)^{np'}
		u^{-p'}(t)dt\bigg)^{1/p'}\geq  \\
	\geq \bigg(2^{-np'} \int_0^{x_1/2} u^{-p'}(t)dt\bigg)^{1/p'}\geq
		\frac{2^{-n}}{(DK)^{1/p'}}
		\bigg(\int_0^r u^{-p'}(t)dt\bigg)^{1/p'}.
\end{gather*}

        Getting back to the original extremal problem (3.7), we see that for
every $r>0$ there exists on $(0,r)$ a function $f_0(t)=u^{-p'}(t)g_0(t)$
such that
\begin{gather*}
	\int_0^r f_0(t)dt\geq \wt{K}(D,n,p)\bigg(\int_0^r u^{-p'}(t)dt
		\bigg)^{1/p'}\bigg(\int_0^r |u(t)f_0(t)|^p dt\bigg)^{1/p},  \\
	\int_0^r t^k f_0(t)\,dt=0, \quad k=1,\dots,n.
\end{gather*}
Extending $f_0$ to $[r,\infty)$ by zero and substituting the obtained function into
the estimate (i) of Theorem 2.1, we get (3.6) with $C=1/\wt{K}$.
\end{proof}

%        LEMMA 3.4.
\begin{lem}
Let $1\leq p\leq q\leq \infty$, $u,v$ be non-negative on
$\bR^+$ functions. The estimate
%	(3.11)
\begin{equation}
	\bigg\|\int_0^x f(t)\,dt|L_{q,v}\bigg\|\leq c\|f|L_{p,u}\|
\end{equation}
holds if and only if
$$	S=\sup_{r>0}\bigg(\int_r^\infty v^q(x)\,dx\bigg)^{1/q}
		\bigg(\int_0^r u^{-p'}(x)\,dx\bigg)^{1/p'}<\infty.	$$
Moreover, if $\wt{c}$ is the best constant in $(3.11)$, then
$S\leq \wt{c}\leq (q')^{1/p'}q^{1/q}S$. If $p=1$ or $q=\infty$, then
$\wt{c}=S$.
\end{lem}

        This is the well-known criterion for the validity of the weighted
Hardy's inequality obtained by J.~S.~Bradley \cite{2} and V.~M.~Kokilashvili \cite{3}.
The proof can be found in \cite{13}, \S 1.3.1.
\vskip+0.2cm

%	PROOF	
\noindent {\it Proof of Theorem $2.1$.}
 Implication (ii)$\RA$(i) is obvious. Equivalence
(ii)$\LRA$(iii) follows from Lemma 3.4. To complete the proof we have to show that
(i)$\RA$(iii).

        We act by induction in $n$. For $n=0$ the assertion of the
theorem follows from Lemma 3.4. Assume that the theorem is
already proved for $n=n_0$ and that
the following estimate holds:
%	(3.12)
\begin{equation}
	\bigg\|\sum_{k=0}^{n_0+1} a_k(x)\int_0^x t^k f(t)\,dt|L_{q,v}\bigg\|
		\leq c\|f|L_{p,u}\|,
\end{equation}
where the function $u$ satisfies the condition (2.2) with $n=n_0+1$.
Lemma 3.3 implies
that $S_0<\infty$. By Lemma 3.4 this is equivalent to the fact that the inequalities
%	(3.13)
\begin{equation}
	\bigg\|a_0(x)\int_0^x f(t)\,dt|L_{q,v}\bigg\|\leq c'\|f|L_{p,u}\|
\end{equation}
are fulfilled. Inequalities (3.12) and (3.13) yield
$$	\bigg\|\sum_{k=1}^{n_0+1} a_k(x)\int_0^x t^k f(t)\,dt|L_{q,v}\bigg\|
		\leq c''\|f|L_{p,u}\|,			$$
which by substitution $\wt{f}(t)=tf(t)$ reduces to the form
$$	\bigg\|\sum_{k=0}^{n_0} a_{k+1}(x)\int_0^x t^k
		\wt{f}(t)\,dt|L_{q,v}\bigg\|
		\leq c''\|\wt{f}|L_{p,\wt{u}}\|,	$$
where $\wt{u}(x)=x^{-1}u(x)$. Since $\wt{u}$ satisfies
the condition (2.2) with $n=n_0$, by assumption of
the induction we obtain the finiteness of the remaining constants $S_k$,
$k=1,\dots,n_0+1$.
\;\;\qed

%        4.
\section{Proof of Theorem 2.2}

%        LEMMA 4.1.
\begin{lem}
Let $[a,b]$ be a segment in $\bR$. For any set of functions
$h_1,\dots,h_l\in L_1([a,b])$
there exists a function $\sg$ with $|\sg(x)|=1$ on $[a,b]$ such that
$$	\int_a^b h_k(x)\sg(x)\,dx\-, \quad k=1,\dots,l.		$$
\end{lem}

        The proof of this lemma can be found in  \cite{14}, p. 267.

%        LEMMA 4.2
\begin{lem}[{\cite{15}}]
 Let $l\in \bN$, $w(x)$ be a non-negative on the segment $[a,b]$
function with $\int_a^b wdx<\infty$. There exists a function $g(x)$,
$x\in \bR$, such that

        \rom{(a)} $g(x)=0$ for $x\notin [a,b]$,

        \rom{(b)} $g,g',\dots,g^{(l-1)}$ are absolutely continuous on $\bR$,

        \rom{(c)} $|g^{(l)}(x)|=w(x)$, $x\in [a,b]$.
\end{lem}

%        PROOF.
\begin{proof}
By Lemma 4.1 there exists a function $\sg$ with $|\sg(x)|=1$ on $[a,b]$
such that
$$	\int_a^b x^k w(x)\sg(x)\,dx=0, \quad k=0,\dots,l-1.	$$
Then
$$	g(x)=	\begin{cases}
			\displaystyle \frac{1}{(l-1)!}\int_a^x (x-t)^{l-1}
				w(t)\sg(t)\,dt, \quad x\in [a,b],  \\
			0, \quad x\notin [a,b],
		\end{cases}		$$
is the required function.
\end{proof}
\vskip+0.2cm

%	PROOF
\noindent {\it Proof of Theorem $2.2$.}

        {\em Step $1$.} Let $T_f(x)=\sum\limits_{k=0}^{l-1} f^{(k)}(0)
		x^k/k!$ be the degree-($l-1$) Taylor polynomial of the function $f$. Then
\begin{gather*}
	P(x,D)f(x)=P(x,D)\bigg( T_f(x)+\int_0^x
		\frac{(x-t)^{l-1}}{(l-1)!}\,f^{(l)}(t)\,dt\bigg)=  \\
	=P(x,D)T_f(x)+\sum_{m=0}^l b_m(x) D^m \int_0^x
		\frac{(x-t)^{l-1}}{(l-1)!}\,f^{(l)}(t)\,dt=  \\
	=P(x,D)T_f(x)+\sum_{m=0}^{l-1} b_m(x) \int_0^x
		\frac{(x-t)^{l-m-1}}{(l-m-1)!}\,
		f^{(l)}(t)\,dt+b_l(x)f^{(l)}(x)= \\
	=P(x,D)T_f(x)+\sum_{m=0}^{l-1} \sum_{k=0}^{l-m-1} b_m(x)\,
		\frac{x^{l-m-1-k}(-1)^k}{(l-m-1-k)!k!} \int_0^x
		t^kf^{(l)}(t)\,dt+  \\
	+ b_l(x)f^{(l)}(x).
\end{gather*}
Changing in the last expression the order of summation with respect to $m$ and
$k$, we arrive at the formula
%	(4.1)
\begin{gather}
	P(x,D)f(x)= P(x,D)T_f(x)+  \notag  \\
	+ \sum_{k=0}^{l-1} \frac{(-1)^k}{k!}\,d_k(x) \int_0^x
		t^k f^{(l)}(t)dt+b_l(x) f^{(l)}(x).
\end{gather}

        {\em Step $2$ $($Sufficiency$)$.}
 From the identity (4.1) there follows the
estimate
%	(4.2)
\begin{gather}
	\|P(x,D)f(x)|L_{q,v}\| \leq
		c\sum_{k=0}^{l-1} |f^{(k)}(0)|\cdot
		\|P(x,D)x^k|L_{q,v}\|+  \notag  \\
	+c\sum_{k=0}^{l-1} \bigg\| d_k(x)\int_0^x
		t^k f^{(l)}(t)\,dt|L_{q,v}\bigg\|+
		\|b_lf^{(l)}|L_{q,v}\|.
\end{gather}

        If conditions (2.3) are fulfilled, then by Lemma 3.4 the second term
in the right-hand side of (4.2) can be bounded by $c'\|f^{(l)}|L_{p,u}\|$.
The third term for $p=q$ can be
estimated in an obvious manner:
$$	\|b_lf^{(l)}|L_{p,v}\|\leq \|b_l vu^{-1}|L_\infty(\bR^+)\|
		\cdot \|f^{(l)}|L_{p,u}\|.		$$
All these arguments imply the sufficiency of conditions in
both parts of Theorem 2.2.

        {\em Step $3$ $($Necessity$)$.}
 The necessity of the condition (2.5) in item (ii)
is obvious. The theorem will be proved if we show that from
%	(4.3)
\begin{equation}
	P(x,D):\os{\circ}{W}{}_{p,u}^{(l)}\to L_{q,v}	
\end{equation}
there follows the fulfillment of (2.3) and (2.4).

        We start with (2.4). Consider the set
$A_\al=\{x\in \bR^+:|b_l vu^{-1}(x)|\geq \al\}$. Suppose
$\mes A_\al>0$. Let $B$ be
a bounded subset of $A$ of positive measure on which the functions
$v(x)d_k(x)$ are bounded, say
$B\sbs [0,M]$, $|v(x)d_k(x)|\leq N$, $x\in B$,
$k=0,\dots,l-1$. For every $\ve>0$ there exists a segment
$\Dl_\ve\sbs \bR^+$ of length $\ve$ such that $\mes \Dl_\ve\cap B>0$. By
Lemma 4.2, there exists an $l$ times differentiable function $g$ supported on
$\Dl_\ve$ and such that
$$	|g^{(l)}(x)|=
		\begin{cases}
			u^{-p'}(x) & \text{if}\;\;x\in \Dl_\ve\cap B,  \\
			0 & \text{otherwise}.
		\end{cases}				$$
The inequality
$$	\|P(x,D)g|L_{q,v}\|\leq K\|g^{(l)}|L_{p,u}\|		$$
and the identity (4.1) yield the estimate
%	(4.4)
\begin{equation}
	\|b_lf^{(l)}|L_{q,v}\|\leq K\|g^{(l)}|L_{p,u}\|+
		c\sum_{k=0}^{l-1}\bigg\|d_k(x) \int_0^x t^k g^{(l)}(t)\,dt
		|L_{q,v}\bigg\|.
\end{equation}

        Using H\"{o}lder's inequality and taking into account the inclusion
\linebreak $\Dl_\ve\cap B\sbs A_\al$, we can see that
%	(4.5)
\begin{gather}
	\|b_lf^{(l)}|L_{q,v}\|=\bigg(\int_{\Dl_\ve\cap B}
		|b_lu^{-p'}v|^qdx\bigg)^{1/q}\geq   \notag  \\
	\geq (\mes \Dl_\ve\cap B)^{1/q-1/p} \bigg(\int_{\Dl_\ve\cap B}
		|b_lu^{-p'}v|^pdx\bigg)^{1/p}\geq   \notag  \\
	\geq (\mes \Dl_\ve\cap B)^{1/q-1/p} \al \bigg(\int_{\Dl_\ve\cap B}
		u^{-p'}(x)\,dx\bigg)^{1/p}.
\end{gather}

        Further,
%	(4.6)-(4.7)
\begin{gather}
	\|g^{(l)}|L_{p,u}\|=\bigg(\int_{\Dl_\ve\cap B}
		u^{-p'}(x)\,dx\bigg)^{1/p},  \\
\ \hskip-1cm	\bigg\|d_k(x) \int_0^x\! t^k g^{(l)}(t)\,dt|L_{q,v}\bigg\|\!\leq \!
		NM^k\bigg(\int_{\Dl_\ve\cap B} u^{-p'}(t)dt\bigg)\times
		\!(\mes \Dl_\ve\cap B)^{1/q}.
\end{gather}
Substituting (4.5)--(4.7) into (4.4) we obtain
%	(4.8)
\begin{gather}
	\al\leq (\mes \Dl_\ve\cap B)^{1/p-1/q}\times  \notag  \\
	\times \bigg(K+cNM^{l-1}
		\bigg(\int_{\Dl_\ve\cap B} u^{-p'}(t)\,dt\bigg)^{1/p'}
		 (\mes \Dl_\ve\cap B)^{1/q} \bigg).
\end{gather}

        We now pass to the limit $\ve \to 0$. Then for $p<q$ Eq.~(4.8) implies that
$\al=0$, while
for $p=q$ it implies $\al\leq K$. Thus the proof of (2.4) is complete.

        Next, by virtue of the identity (4.1) and also from (4.3) and (2.4) it
follows that the estimate
$$	\bigg\|\sum_{k=0}^{l-1} \frac{(-1)^k}{k!}\,d_k(x)
		\int_0^x t^k f^{(l)}(t)\,dt|L_{q,v}\bigg\|\leq
		\|f^{(l)}|L_{p,u}\|		$$
holds. Now, Theorem 2.1 implies (2.3).
\;\;\qed

\section*{Acknowledgement}

The work is carried out under the financial support of Russian Fund of
Fundamental Investigations, Grant RFFI--96-01-00243. {\bf Added June 2021:} The author is grateful for the hospitality to the Balabanovo bootcamp facility, where this work was carried out in the summer of 1996.

\end{document}